\documentclass[11pt,a4paper]{article}

\usepackage[utf8]{inputenc}
\usepackage{amsmath}
\usepackage{amssymb}
\usepackage{amsthm}
\usepackage{hyperref}

\usepackage{orcidlink}
\usepackage{authblk}
\usepackage{enumitem}

\DeclareMathOperator{\Hom}{Hom}

\DeclareMathOperator{\ud}{ud}

\def\udw{\  \ensuremath{{\bf\ud_w}}\  }

\newcommand{\T}{\mathbb{T}}
\newcommand{\Z}{\mathbb{Z}}
\newcommand{\N}{\mathbb{N}}

\def\sub{\subseteq}

\def\into{\longrightarrow}
\def\sub{\subseteq}
\def\setdiff{\backslash}

\def\dualGamma{\widehat{\Gamma}}

\def\higherW{\mathcal{W}}
\def\higherN{\mathfrak{N}}

\def\arbGroup{\mathcal{G}}

\def\KroU{\mathbb{U}}

\newtheorem{theorem}{Theorem}[section]
\newtheorem{lemma}[theorem]{Lemma}

\title{A Kronecker-Weyl Theorem from Interpolation Sets}
\author{Rafael Reno S. Cantuba\thanks{Associate Professor, Department of Mathematics and Statistics, De La Salle University (DLSU), Taft Ave., Manila, Philippines, supported by a grant from the Reserach Grants Management Office of DLSU, grant no.: 02FR1TAY24-1TAY25, email: rafael.cantuba@dlsu.edu.ph}\\ ORCID: 0000-0002-4685-8761\orcidlink{0000-0002-4685-8761}}
\date{}

\begin{document}

\maketitle

\begin{abstract}
An analog of the classical Kronecker–Weyl theorem for weak uniform distribution is obtained for an arbitrary countably infinite independent subset of a discrete abelian group using a type of interpolation sets called $\varepsilon$-Kronecker sets. Specifically,  the topological size of the set of homomorphisms inducing weakly uniformly distributed sequences is established by showing that it forms a dense $G_\delta$ set.
\end{abstract}

\vspace{0.5em}
\noindent
\begin{quote}
    
    \vspace{0.5em}
\end{quote}

\section{Introduction}

The theory of Diophantine approximation is said to seek to understand how well, that is, how closely, real numbers can be ``trapped'' by relations with the integers \cite[p.~1]{hla91}. Some key results in Diophantine approximations are the classical approximations of Dirichlet and Kronecker. The latter may be proven using the former \cite[Chapter~2]{hla91}, and this ``Kronecker Approximation Theorem'' may also be stated in terms of uniform distribution of sequences \cite[Chapter~1]{kui74}, or more precisely, the ``Weyl criterion'' \cite[p.~4777]{dik11}, and hence, we further have the ``Kronecker-Weyl Theorem'' for the integers. Motivation may be taken from \cite{dik11}  on how topological groups may be used to obtain generalizations of the Kronecker-Weyl Theorem. Topological groups are the main vehicle in abstract harmonic analysis \cite[p.~v]{rud62}, and we saw it as an opportunity to further explore the use of topological groups and abstract harmonic analysis in Diophantine approximations.

Throughout, any adjective that appears before ``group'' or ``abelian group'' refers to the underlying topology that makes the group a topological group, which means that the group operation and inversion are both continuous. If we let $C(\arbGroup)$ be the collection of all continuous complex-valued functions on a topological group $\arbGroup$, and if we let $\Hom(\arbGroup,\T)$ be the collection of all group homomorphisms of $\arbGroup$ into the compact group $\T$ of all complex numbers with modulus one, then the dual group of $\arbGroup$ is $\widehat{\arbGroup}:=C(\arbGroup)\cap\Hom(\arbGroup,\T)$. 

A sequence $(x_n)$ in a compact abelian group $K$ is \emph{uniformly distributed} if, for each nontrivial character $\chi$ in the dual group of $K$, 
\[
\lim_{N\rightarrow\infty}\frac{1}{N}\sum_{n=1}^N\chi(x_n)=0.
\]
In symbols, we write: $(x_n)\ud K$.
Given a discrete abelian group $\Gamma$, and a sequence $(a_n)$ in $\Gamma$, if $E\sub\Gamma$ is the collection of all terms in the sequence $(a_n)$, then
\[
\KroU_0(E,K):=\{h\in\Hom(\Gamma,K)\  :\  (h(a_n))\ud K\}.
\]
Following \cite[p.~4777]{dik11}, the setting of the classical Kronecker-Weyl theorem, in metric number theory and Diophantine approximations, is about how ``large'' the set $\KroU(E,K)$ is when $\Gamma$ is the discrete group $\Z$ of all integers, and when $K=\T$.

In the study \cite{dik11}, less preference was given to the sets $\KroU_0(E,K)$ in favor of some ``natural topological counterpart'' of these sets because for the former, there is no natural order on an arbitrary abelian group that allows us to index $E$ \cite[p.~4778]{dik11}.

We show in this work, however, that a Kronecker-Weyl-type theorem may still be achieved in this direction once we incorporate results from harmonic analysis. In particular, we shall be using conditions on $E$ that make it an ``interpolation set'' in the sense of \cite{gra13}. More precisely, a sufficient condition we have for $\KroU_0(E,K)$ to be a dense $G_\delta$ set involves the presence in $E$ of what are called $\varepsilon$-Kronecker sets, a localized form of Kronecker sets, which are naturally related to the classical Kronecker approximation theorem \cite[p.~43]{gra13}.

For simpler arguments, we shall make use in this paper of a weaker version of uniform distribution in discrete abelian groups, which is the following. A sequence $(x_n)$ in a compact abelian group $K$ is \emph{weakly uniformly distributed} if, for each nontrivial character $\chi$ in the dual group of $K$, and each positive integer $k$, there exists a positive integer $N$ such that
\[
\left|\frac{1}{N}\sum_{n=1}^N\chi(x_n)\right|<\frac{1}{k},
\]
which we write in symbols as $(x_n)\udw K$. The corresponding sets are
\[
\KroU(E,K):=\{h\in\Hom(\Gamma,K)\  :\  (h(a_n))\udw K\}.
\]
Thus, $\KroU_0(E,K)\sub\KroU(E,K)$, and establishing the topological size of $\KroU(E,K)$ provides a natural Kronecker-Weyl-type theorem for weak uniform distribution. We prove that $\KroU(E,\T^\T)$ is a dense $G_\delta$ set in this setting.

\section{Preliminaries}\label{PrelSec}

By an \emph{arc} along $\T$ we shall mean the image, under the mapping $s\mapsto e^{is}$, of an open interval $I\sub[-\pi,\pi)$, with length less than $\pi$. The assumption on length apparently seems unnecessary for the intuitive notion of arc, but it shall serve a purpose in an argument that shall come shortly. Let $Q$ be a neighborhood of the identity $1\in\T$, which we assume, without loss of generality, to be an arc. By the continuity of an arbitrary nontrivial character $\chi\in\widehat{\T^\T}$, the set $\chi^{-1}[Q]$ contains a neighborhood $\higherW$ of the identity in $\T^\T$, which is the constant function $1:e^{i\theta}\mapsto 1$.

Given $t\in\T$ and an arc $U\sub\T$, if $S(t,U)$ is the inverse image of $U$ under the projection map $\T^\T\into\T$ given by $f\mapsto f(t)$, then the collection of all such sets $S(t,U)$ form a subbasis for the product topology on $\T^\T$. Thus, for the basic neighborhood $Q$ of $1\in\T$, there exists a finite subset $F=F(Q)=\{t_1,t_2,\ldots,t_M\}$ of $\T$ such that $\bigcap_{k=1}^MS(t_k,U_k)=\higherW\sub \chi^{-1}[Q]$, where, for each $k\in\{1,2,\ldots,M\}$, we have $1=1(t_k)\in U_k$. By a routine argument, $\higherN_F:=\{f\in\T^\T\  :\  f[F]=\{1\}\}\sub S(t_k,U_k)$ for all $k$, and we further have $\higherN_F\sub\chi^{-1}[Q]$. Applying $\chi$ to both sides, we further obtain $\chi[\higherN_F]\sub Q$.

The earlier assumption that any arc has (its inverse image under $s\mapsto e^{is}$ of) length less than $\pi$ implies that the proper subset $Q$ of $\T$ does not contain a nontrivial subgroup of $\T$, and this is often referred to in the literature as the property of $\T$ that it has \emph{no small subgroups}. See, for instance, \cite[Section~II.1]{kap71}. The set $\higherN_F$ is a subgroup of $\T^\T$, and since $\chi$ is a group homomorphism and $\chi[\higherN_F]\sub Q$, the only possibility is $\chi[\higherN_F]=\{1\}$.

Given $x\in\T^\T$, and $t_k\in\{t_1,t_2,\ldots,t_M\}=F$, if we define $w_k\in\T^\T$ by the rule that $w_k(t)=x(t)$ if $t=t_k$ and otherwise $w_k(t)=1$, then by a routine argument, evaluating the function $x(w_1 w_2\cdots w_M)^{-1}$ at an arbitrary $t\in F$ leads to $x(w_1 w_2\cdots w_M)^{-1}\in\higherN_F$. Since $\chi[\higherN_F]=\{1\}$, we have $\chi(x(w_1 w_2\cdots w_M)^{-1})=1$. Since $\chi$ is a group homomorphism, $\chi(x)=\chi(w_1)\chi(w_2)\cdots\chi(w_M)$. 

Given $k\in\{1,2,\ldots,M\}$, define $c_k:\T\into \T^\T$ by the rule that $c_k(\theta)$ is the function that sends $t\in\T$ to $\theta$ if $t=t_k$ and to $1$ otherwise. By routine arguments, we have the function equality $w_k = c_k(x(t_k))$ for each $k\in\{1,2,\ldots,M\}$. Also by routine arguments, $\theta\mapsto(\chi\circ c_k)(\theta)$ is a continuous group homomorphism of $\T$, or is a character in $\widehat{\T}$, so there exists an integer $n_k$ such that $(\chi\circ c_k)(\theta)=\theta^{n_k}$ for all $\theta\in\T$. Evaluating this character identity at the specific value $\theta = x(t_k)$ yields $\chi(w_k) = \chi(c_k(x(t_k))) = [x(t_k)]^{n_k}$ for each $k$.

At this point, we have proven that for each nontrivial character $\chi \in \widehat{\T^\T}$, there exist a finite subset $F = \{t_1, t_2, \ldots, t_M\}$ of $\T$ and nonzero integers $n_1, n_2, \ldots, n_M$ such that 
\begin{flalign}
    && \chi(x) &=[x(t_1)]^{n_1}[x(t_2)]^{n_2}\cdots[x(t_M)]^{n_M}, & (x \in \T^\T).\label{chiproduct0}
\end{flalign}

We fix a countable dense subset $\T_0$ of $\T$. The collection of all characters $\chi \in \widehat{\T^\T}$ which, when decomposed as in \eqref{chiproduct0}, possess the property that the associated points $t_1, t_2, \ldots, t_M$ belong to $\T_0$, is countable. Consequently, we may enumerate these characters as a sequence $(\chi_\mu)$.

For the discrete abelian group $\Gamma$, suppose that, for some indexing set $\Omega$, the collection $\{U_m\  :\  m\in\Omega\}$ is a basis for the topology on $\dualGamma^\T$. For each finite subset $\Lambda$ of $\T\times\Omega$, we define
\begin{equation}
  \higherW(\Lambda):=\{h\in\dualGamma^\T\  :\  (t,m)\in\Lambda\implies h(t)\in U_m\} .\label{Wbasis}
\end{equation}
Since $\Gamma$ is discrete, $\dualGamma=C(\dualGamma)\cap\Hom(\dualGamma,\T)=\Hom(\dualGamma,\T)$. We further have the topological group isomorphism $\dualGamma^\T \cong \Hom(\Gamma,\T^\T)$, and for the topology of this space, the collection of all sets \eqref{Wbasis} is a basis.

Suppose $E$ is countably infinite subset of $\Gamma$,  and the sequence $(e_n)$ an enumeration of the elements of $E$. Since $\T_0$ is also countably infinite, we fix a sequence $(t_n)$ that is an enumeration of the elements of $\T_0$. 

Furthermore, for each character $\chi_\mu$ decomposed as in \eqref{chiproduct0} with points $\tau_1, \dots, \tau_M \in \T_0$, we choose a map $\phi_\mu : \{\tau_1, \ldots, \tau_M\} \into E$ that assigns to each evaluation point $\tau_m \in \T_0$ a corresponding element $g_m := \phi_\mu(\tau_m) \in E$.

Thus, for any character $y \in \dualGamma$, \eqref{chiproduct0} becomes
\begin{flalign}
    && \chi_\mu(y\circ\phi_\mu) &=[y(g_1)]^{n_1}[y(g_2)]^{n_2}\cdots[y(g_M)]^{n_M}, & (\mu\in\N,\  y \in \dualGamma).\label{chiproduct}
\end{flalign}

\section{Weakly uniformly distributed sequences in a discrete abelian group}

We now give a decomposition of the set $\KroU(E,\T^\T)$ that shall aid in proving its size. For each triple $(N,\mu,k)\in\N^3$, we define
\begin{equation}
O_{N, \mu, k} := \left\{ h \in \dualGamma^\T \ :\  \left| \frac{1}{N} \sum_{n=1}^{N} \chi_{\mu}\big(h(t_n)\circ\phi_\mu\big) \right| < \frac{1}{k} \right\}.
\label{Odef_clean}
\end{equation}
Under the aforementioned choices of $t_n$ and $\phi_\mu$, for any $h \in \dualGamma^\T$ and its image $\widetilde{h} \in \Hom(\Gamma, \T^\T)$ under the topological group isomorphism $\dualGamma^\T \cong \Hom(\Gamma, \T^\T)$ defined by $\widetilde{h}(e)(t) = h(t)(e)$, we have
\[
\chi_\mu\big(h(t_n) \circ \phi_\mu\big) = \prod_{m=1}^M \big[ h(t_n)(\tau_m) \big]^{n_m} = \prod_{m=1}^M \big[ \widetilde{h}(\tau_m)(t_n) \big]^{n_m} = \chi_\mu\big(\widetilde{h}(e_n)\big).
\]
Thus, the set $O_{N, \mu, k}$ has an alternative characterization in $\Hom(\Gamma, \T^\T)$, which is
\begin{equation}
O_{N, \mu, k} = \left\{ h \in \Hom(\Gamma, \mathbb{T}^\mathbb{T}) \ :\ \left| \frac{1}{N} \sum_{n=1}^{N} \chi_{\mu}\big(h(e_n)\big) \right| < \frac{1}{k} \right\}.
\label{Odef_hom_version}
\end{equation}

\begin{lemma}\label{MainLem1}  $\KroU(E,\T^\T)=\bigcap_{\mu=1}^\infty\bigcap_{k=1}^\infty\bigcup_{N=1}^\infty O_{N,\mu,k}.$
\end{lemma}
\begin{proof} 
The set inclusion ``$\sub$'' follows immediately from the definition of the sets $O_{N,\mu,k}$ and the definition of a uniformly distributed sequence. To proceed with the other set inclusion, let $h\in\bigcap_{\mu=1}^\infty\bigcap_{k=1}^\infty\bigcup_{N=1}^\infty O_{N,\mu,k}$, let $\chi\in\widehat{\T^\T}$ be a nontrivial character, and let $k\in\N$. We decompose $\chi$ as the finite product \eqref{chiproduct0}. 

For each $x \in \mathbb{T}^\mathbb{T}$, the map $\Phi_x : \mathbb{T}^M \to \mathbb{T}$ that sends $(s_1, \dots, s_M)$ to $\prod_{m=1}^M [x(s_m)]^{n_m}$ is continuous at $(t_1, \dots, t_M)$. Thus, there exists a neighborhood $Q_x \subseteq \mathbb{T}^M$ of $(t_1, \dots, t_M)$ such that
\begin{flalign}
    && |\Phi_x(t_1, \dots, t_M) - \Phi_x(s_1, \dots, s_M)| &< \frac{1}{2k}, & ((s_1, \dots, s_M) \in Q_x).\label{TMnhood}
\end{flalign}
Since $\mathbb{T}_0^M$ is dense in $\mathbb{T}^M$, there exists a sequence $(T_s)_{s=1}^\infty$ in $\T_0^M$ that converges to $(t_1,\ldots,t_M)$. If $\pi_m:\T_0^M\to\T_0$ is the projection onto the $m$th coordinate, then the character $\chi_{\mu_s}$ defined by
\begin{flalign}
    && \chi_{\mu_s}(x) &=[x(\pi_1(T_s))]^{n_1}[x(\pi_2(T_s))]^{n_2}\cdots[x(\pi_M(T_s))]^{n_M}, & (x \in \mathbb{T}^\mathbb{T}),\label{chiproduct0B}
\end{flalign}
is hence one of the characters in the sequence $(\chi_\mu)$. By the definition of $h$, for each $s \in \N$, there exists $N_s \in \mathbb{N}$ such that
\begin{equation}
\left| \frac{1}{N_s} \sum_{\nu=1}^{N_s} \chi_{\mu_s}(h(e_\nu)) \right| < \frac{1}{2k}.\label{hsum}
\end{equation}

For each $s \in \mathbb{N}$, the finite intersection $O_s = \bigcap_{\nu=1}^{N_s} Q_{h(e_\nu)}$ is an open neighborhood of $(t_1, \dots, t_M)$. Since $(T_s)$ converges to $(t_1,\ldots,t_M)$, there exists some index $\sigma \in \mathbb{N}$ such that $T_\sigma \in O_\sigma$. Hence, $T_\sigma \in Q_{h(e_\nu)}$ for all $\nu \in \{1, \ldots, N_\sigma\}$. Setting $x=h(e_\nu)$ and $s_m=\pi_m(T_\sigma)$ in \eqref{TMnhood} for all $m\in\{1,2,\ldots,M\}$, we obtain
\begin{flalign}
    && |\chi(h(e_\nu)) - \chi_{\mu_\sigma}(h(e_\nu))| &< \frac{1}{2k},& (\nu \in \{1, \dots, N_\sigma\}). \label{hsum2}
\end{flalign}

With the aid of \eqref{hsum} (applied to $s=\sigma$) and \eqref{hsum2}, we further have
\begin{flalign}
&& \left| \frac{1}{{N_\sigma}} \sum_{\nu=1}^{N_\sigma} \chi(h(e_\nu)) \right| &= \left| \frac{1}{{N_\sigma}} \sum_{\nu=1}^{N_\sigma} \Big( \chi(h(e_\nu)) - \chi_{\mu_\sigma}(h(e_\nu)) + \chi_{\mu_\sigma}(h(e_\nu)) \Big) \right|, & \nonumber\\
&& &\le \left| \frac{1}{{N_\sigma}} \sum_{\nu=1}^{N_\sigma} \Big( \chi(h(e_\nu)) - \chi_{\mu_\sigma}(h(e_\nu)) \Big) \right|&\nonumber\\
&& &\quad + \left| \frac{1}{{N_\sigma}} \sum_{\nu=1}^{N_\sigma} \chi_{\mu_\sigma}(h(e_\nu)) \right|,& \nonumber\\
&& &\le \frac{1}{{N_\sigma}} \sum_{\nu=1}^{N_\sigma} |\chi(h(e_\nu)) - \chi_{\mu_\sigma}(h(e_\nu))| + \left| \frac{1}{{N_\sigma}} \sum_{\nu=1}^{N_\sigma} \chi_{\mu_\sigma}(h(e_\nu)) \right|,&\nonumber\\
&& &<\frac{1}{{N_\sigma}}\sum_{\nu=1}^{N_\sigma}\frac{1}{2k}+\frac{1}{2k}=\frac{1}{k}.&\nonumber
\end{flalign}
Therefore, $h\in\KroU(E,\T^\T)$, and this completes the proof.
\end{proof}

\section{Localized Kronecker sets}

The sufficient condition for the density of $\KroU(E,\T^\T)$ that we shall prove in this paper involves a kind of interpolation sets from harmonica analysis called $\varepsilon$-Kronecker sets, which are ``localized'' forms of Kronecker sets that have a natural relation to the Kronecker Approximation Theorem \cite[Section~2.8]{gra13}.

Given $U\sub\widehat{\Gamma}$ and $\varepsilon>0$, a subset $X$ of $\Gamma$ is a \emph{(weak) $\varepsilon$-Kronecker$(U)$ set} if for each $\varphi\in\T^X$, there exists $x\in U$ such that if $\gamma\in X$, then $|\varphi(\gamma)-x(\gamma)|$ is (at most, or respectively) strictly less than $\varepsilon$. A (weak) $\varepsilon$-Kronecker$(\dualGamma)$ set is called a \emph{(weak) $\varepsilon$-Kronecker set}.

A subset $X \subseteq \Gamma \setdiff \{1\}$ is an \emph{independent set} if for every finite subset $F \subseteq X$ and any collection of integers $\{m_\gamma : \gamma \in F\}$, the condition $\prod_{\gamma \in F} \gamma^{m_\gamma} = 1$ is true if and only if $\gamma^{m_\gamma} = 1$ for all $\gamma\in F$.

The classical Kronecker Approximation Theorem may be stated as the assertion that every independent set of real numbers is a $\varepsilon$-Kronecker set for any $\varepsilon>0$ \cite[p.~43]{gra13}. For the more general setting of a discrete abelian group and its compact dual group, we have the following.

\begin{lemma}[{\cite[Corollary~2.2.10]{gra13}}]\label{premiseLem1} If $E$ is an independent set and if every element of $E$ has infinite order, then $E$ is an $\varepsilon$-Kronecker set for any $\varepsilon>0$.
\end{lemma}

Since an $\varepsilon$-Kronecker set is a weak $\varepsilon$-Kronecker set, the previous lemma may be used in conjunction with the following to obtain a condition involving the ``local'' notion of $\varepsilon$-Kronecker$(U)$ set.

\begin{lemma}[{\cite[Theorem~2.2.13]{gra13}}]\label{premiseLem2} If $E$ is a weak $\varepsilon$-Kronecker set for any $\varepsilon>0$, then for each $\varepsilon>0$ and every open set $U\sub\widehat{\Gamma}$, there exists a finite set $F\sub\Gamma$ such that $E\setdiff F$ is weak $\varepsilon$-Kronecker$(U)$.
\end{lemma}

The conclusion of the previous lemma shall be instrumental in the proof of the following.

\begin{lemma}\label{MainLem2} If the countably infinite subset $E$ of $\Gamma$ is independent and every element of $E$ has infinite order, then for every $(N,\mu,k)\in\N^3$ and every finite $\Lambda\sub\T\times\Omega$,
\[\higherW(\Lambda)\cap O_{N,\mu,k}\neq\emptyset.\]
\end{lemma}

\begin{proof}
Let $(N,\mu,k)\in\N^3$, and consider a finite $\Lambda \sub \T \times \Omega$. Since the sets in the sequence $U_m$ form a basis for the topology on $\dualGamma$, the finite intersection
\[
Q:=\bigcap_{(t,m)\in\Lambda} U_{m}
\]
is an open subset of $\dualGamma$. 

By the definition of the character $\chi_\mu$, there exist $\tau_1,\tau_2,\ldots,\tau_M\in\T_0$ and nonzero integers $n_1,n_2,\ldots,n_M$ such that
\begin{flalign}
&&    \chi_\mu(x) &=[x(\tau_1)]^{n_1}[x(\tau_2)]^{n_2}\cdots[x(\tau_M)]^{n_M}, &(x\in\T^\T).\label{Omembership0}
\end{flalign}
Let $\varepsilon := \frac{1}{k\sum_{m=1}^M|n_m|} > 0$. Because $E$ is independent and every element of $E$ has infinite order, by routine use of Lemmas~\ref{premiseLem1}--\ref{premiseLem2}, there exists a finite set $F\sub\Gamma$ such that $E\setdiff F$ is a weak $\frac{1}{2}\varepsilon$-Kronecker$(Q)$ set. 

Since $F$ is finite, we reindex the sequence representation $(e_n)_{n=1}^\infty$ of $E$ to exclude $F$, and adjust the assignment map $\phi_\mu : \{\tau_1, \ldots, \tau_M\} \into E $ accordingly so that $g_m = \phi_\mu(\tau_m) \in E \setdiff F$ for all $m \in \{1, \ldots, M\}$. Consequently, \eqref{Omembership0} becomes
\begin{flalign}
&&    \chi_\mu(y\circ\phi_\mu) &=[y(g_1)]^{n_1}[y(g_2)]^{n_2}\cdots[y(g_M)]^{n_M}, &(y\in\dualGamma).\label{Omembership}
\end{flalign}

For each $\nu\in\{1,2,\ldots,N\}$, we define $\varphi_\nu:E\setdiff F\into\T$ by 
\[\varphi_\nu(e):=\begin{cases} e^{i\frac{2\pi \nu}{Nn_1}}, & e=g_1,\\
    1, &e\neq g_1.
\end{cases}\]
Since $E\setdiff F$ is a weak $\frac{1}{2}\varepsilon$-Kronecker$(Q)$ set, for each $\nu\in\{1,2,\ldots,N\}$, there exists a character $\gamma_\nu\in Q$ such that
\begin{flalign}
   && |\varphi_\nu(g_m)-\gamma_\nu(g_m)|&\leq\frac{1}{2}\varepsilon<\varepsilon,&(m\in\{1,2,\ldots M\}).\label{boundbyT0}
\end{flalign}

For each $\nu\in\{1,2,\ldots,N\}$ and each $m\in\{1,2,\ldots,M\}$, if we define
\begin{flalign}
    && \alpha_\ell=\alpha_\ell(m,\nu)&:=[\varphi_\nu(g_m)]^{n_m-1-\ell}[\gamma_\nu(g_m)]^{\ell},&(\ell\in\{0,\ldots,n_m-1\}),\label{boundbyT1}\\
    && \beta_m=\beta_m(\nu)&:=\prod_{\ell=1}^{m-1}[\gamma_\nu(g_\ell)]^{n_\ell},&\label{boundbyT2}\\
    && \omega_m=\omega_m(\nu)&:=\prod_{\ell=m+1}^{M}[\varphi_\nu(g_\ell)]^{n_\ell},&\label{boundbyT3}
\end{flalign}
then we have the identities
\begin{flalign}
    && [\varphi_\nu(g_m)]^{n_m}-[\gamma_\nu(g_m)]^{n_m} &=[\varphi_\nu(g_m)-\gamma_\nu(g_m)]\sum_{\ell=0}^{n_m-1}\alpha_\ell,&\label{boundbyT4}\\
    && \prod_{m=1}^M[\varphi_\nu(g_m)]^{n_m}-\prod_{m=1}^M[\gamma_\nu(g_m)]^{n_m} &=\sum_{m=1}^M\beta_m\omega_m[\varphi_\nu(g_m)]^{n_m}&\nonumber\\
    && &\quad-\sum_{m=1}^M\beta_m\omega_m[\gamma_\nu(g_m)]^{n_m}.&\label{boundbyT5}
\end{flalign}
Since $\varphi_\nu$ and $\gamma_\nu$ are $\T$-valued functions, from \eqref{boundbyT1}--\eqref{boundbyT3}, we find that $|\alpha_\ell|$, $|\beta_m|$ and $|\omega_m|$ are all at most $1$, for all $\ell$ and $m$. Thus, we obtain from \eqref{boundbyT4} the inequality $|[\varphi_\nu(g_m)]^{n_m}-[\gamma_\nu(g_m)]^{n_m}|\leq |n_m|\cdot|\varphi_\nu(g_m)-\gamma_\nu(g_m)|$, on which we use \eqref{boundbyT0} to obtain $|[\varphi_\nu(g_m)]^{n_m}-[\gamma_\nu(g_m)]^{n_m}| < |n_m|\varepsilon$, which we apply to \eqref{boundbyT5} to obtain
\begin{flalign}
    &&\left|\prod_{m=1}^M[\varphi_\nu(g_m)]^{n_m}-\prod_{m=1}^M[\gamma_\nu(g_m)]^{n_m}\right|&\leq\sum_{m=1}^M|\beta_m|\cdot|\omega_m|\cdot|n_m|\varepsilon&,\nonumber\\
   && &<\varepsilon\sum_{m=1}^M |n_m|=\frac{1}{k},&\nonumber\\
   &&\left|\prod_{m=1}^M[\varphi_\nu(g_m)]^{n_m}-\prod_{m=1}^M[\gamma_\nu(g_m)]^{n_m}\right|&<\frac{1}{k}.&\label{rootsunity1}
\end{flalign}
Using the definition of $\varphi_\nu$,
\[\prod_{m=1}^M[\varphi_\nu(g_m)]^{n_m}=[\varphi_\nu(g_1)]^{n_1}\cdot\prod_{g_m\neq g_1}[\varphi_\nu(g_m)]^{n_m}=[e^{i\frac{2\pi \nu}{Nn_1}}]^{n_1}\cdot \prod_{m=2}^M 1=e^{i\frac{2\pi \nu}{N}},
\]
so \eqref{rootsunity1} further becomes
\begin{equation}
    \left|\prod_{m=1}^M[\gamma_\nu(g_m)]^{n_m}-e^{i\frac{2\pi \nu}{N}}\right|<\frac{1}{k}.\label{rootsunity2}
\end{equation}
In the equation
\[
    \frac{1}{N}\sum_{\nu=1}^N\prod_{m=1}^M[\gamma_\nu(g_m)]^{n_m}=\frac{1}{N}\sum_{\nu=1}^Ne^{i\frac{2\pi \nu}{N}}+\frac{1}{N}\sum_{\nu=1}^N\left[\prod_{m=1}^M[\gamma_\nu(g_m)]^{n_m}-e^{i\frac{2\pi \nu}{N}}\right],
\]
we have $\frac{1}{N}\sum_{\nu=1}^Ne^{i\frac{2\pi \nu}{N}}=0$, and so by \eqref{rootsunity2},
\begin{equation}
   \left|\frac{1}{N}\sum_{\nu=1}^N\prod_{m=1}^M[\gamma_\nu(g_m)]^{n_m}\right| < 0+\frac{1}{N}\sum_{\nu=1}^N\frac{1}{k}=\frac{1}{k}.\label{rootsunity3}
\end{equation}

Let $\pi_1:\T\times\Omega\into\T$ be the projection onto the first coordinate. We define $h\in\dualGamma^\T$ by $h(t)=\gamma_\nu$ if there exists $\nu\in\{1,2,\ldots,N\}$ such that $t=t_\nu$, by $h(t)=\gamma_1$ if $t\in\pi_1(\Lambda)\setdiff\{t_1,t_2,\ldots,t_N\}$, and for all other cases, define $h(t)$ as the identity character from $\dualGamma$. 

Setting $y=h(t_\nu)$ in \eqref{Omembership}, and substituting the result into \eqref{rootsunity3}, proves that $h\in O_{N,\mu,k}$.

If $(t,m)\in\Lambda$, then by the definition of $h$, we have $h(t)=\gamma_\nu$, where $t=t_\nu$ or $\nu=1$ if $t\notin\{t_1,t_2,\ldots,t_N\}$. Since each $\gamma_\nu \in Q$, we have $h(t)\in Q\sub U_m$, so $h\in\higherW(\Lambda)$. 

Therefore, $h\in\higherW(\Lambda)\cap O_{N,\mu,k}$ as desired.
\end{proof}

\section{Summary and final remarks}

Following \cite[Remark~2.6(i)]{dik11}, a dense subspace $X$ of $\dualGamma^\T$ has the \emph{Baire property} if $X$ intersects every set $W\cap\bigcap_{n=1}^\infty Q_n$ where $W$ is a nonempty open subset of $\dualGamma^\T$ and each $Q_n$ is a dense open subset of $\dualGamma^\T$. By a routine argument, a dense $G_\delta$ subset of $\dualGamma^\T$ has the Baire property. We are now ready to use Lemma~\ref{MainLem2} to describe the size of $\KroU(E,\T^\T)$, or equivalently, to give our Kronecker-Weyl Theorem.

\begin{theorem}\label{TheThm}
For each countably infinite independent subset $E$ of a discrete abelian group $\Gamma$, if every element of $E$ has infinite order, then $\KroU(E, \T^\T)$ is dense in $\widehat{\Gamma}^\T$ and has the Baire property.
\end{theorem}
\begin{proof} Given $(N,\mu,k)\in\N^3$, for each $n\in\{1,2,\ldots,N\}$, the function $\alpha_n:\dualGamma^\T\into\T$, defined by $\alpha_n:h\mapsto\chi_{\mu}\big(h(t_n)\circ\phi_\mu\big)$, is continuous. Thus, $O_{N,\mu,k}$ is the inverse image of the open ball on the complex plane centered at $0$ with radius $\frac{1}{k}$ under the continuous function $\frac{1}{N}\sum_{n=1}^N\alpha_n$, so $O_{N,\mu,k}$ is an open set. Since the collection $\{\higherW(\Lambda)\  :\  \Lambda\sub\T\times\Omega,\  |\Lambda|<\infty\}$ is a basis for the topology on $\dualGamma^\T$, and since the hypotheses of Lemma~\ref{MainLem2} are true, $O_{N,\mu,k}$ is dense in $\dualGamma^\T$, and so, by Lemma~\ref{MainLem1}, $\KroU(E,\T^\T)$ is a dense $G_\delta$ subset of $\dualGamma^\T$, and hence has the Baire property.
\end{proof}

The reader may verify by routine arguments that the notion of independent set implies the notion of almost zero-torsion, which was used in \cite{dik11}. However, we leave here as a conjecture as whether the density or simply just the  nonemptiness of $\KroU(E,\T^\T)$ or $\KroU(E,\T^\kappa)$, respectively, for some cardinal $\kappa\geq 1$, implies independence and the condition that $E$ has only elements of infinite order. This is supposed to form the complete analog of \cite[Theorem~3.1]{dik11}, which is a statement of equivalent conditions. A further exploration may be made on how analogs of the corollaries in the rest of \cite[Section~3]{dik11} may be stated and proven for Theorem~\ref{TheThm} above. We note, however, that a routine use of transfinite induction may be used to prove \cite[Corollary~3.6]{dik11} with the use of only Lemmas~\ref{premiseLem1}--\ref{premiseLem2}. Finally, a continuation of this study may involve Bohr topologies and Bohr compactifications, which are explored in both \cite[Section~6]{dik11} and \cite[Section~2.7.1]{gra13}, which means that both $\varepsilon$-Kronecker sets and generalizations of the Kronecker-Weyl Theorem may be studied further in the setting of Bohr topologies.

\section*{Ethical approval} This is not applicable to research work in pure mathematics.

\section*{Funding} The author was supported by the Research Grants Management Office of De La Salle University, Taft Ave., Manila, Philippines, with grant no.: 02FR1TAY24-1TAY25.

\section*{Availability of data and materials} Studies in pure mathematics do not involve data sets, and hence a declaration on availability of data and materials is not applicable.

\end{document}